\documentclass[a4paper,fleqn]{cas-dc}

\usepackage[numbers,sort&compress]{natbib}
\usepackage{amsmath,amssymb,amsthm,mathtools,bm}
\usepackage{thmtools}
\usepackage{booktabs}
\usepackage{algorithm}
\usepackage{algpseudocode}
\usepackage[nameinlink,noabbrev]{cleveref}
\usepackage{optprog}

\newtheorem{theorem}{Theorem}
\newtheorem{lemma}[theorem]{Lemma}
\newtheorem{proposition}[theorem]{Proposition}
\newtheorem{observation}[theorem]{Observation}

\theoremstyle{definition}
\newtheorem{definition}[theorem]{Definition}
\newtheorem{remark}[theorem]{Remark}

\newcommand{\R}{\mathbb{R}}

\renewcommand{\L}{\mathcal{L}}

\begin{document}

\shorttitle{Max-$k$-Cut via Node Features}
\shortauthors{Bhardwaj, Gogoi and Narayanan}

\title[mode=title]{Max-$k$-Cut via Node Features}

\author[1]{Avinash Bhardwaj}[orcid=0000-0002-9690-1705]
\author[1]{Hritiz Gogoi}[orcid=0009-0005-0216-1232]
\author[1]{Vishnu Narayanan}


\affiliation[1]{
  organization={Department of Industrial Engineering and Operations Research,
  Indian Institute of Technology Bombay},
  addressline={Powai},
  city={Mumbai},
  postcode={400076},
  state={Maharashtra},
  country={India}
}

\begin{abstract}
We study the Max-$k$-Cut problem from a node-feature perspective, where each
vertex is associated with a feature vector and edge weights are given by pairwise
inner products. We first examine the semidefinite relaxation of Max-$k$-Cut from
this perspective. Using a normal-cone argument, we derive a general sufficient
condition for exactness of the Frieze--Jerrum relaxation and show that it is
satisfied in two feature-structural regimes: perfect feature balance, where the
aggregate feature vectors of the parts are equal, and feature dominance, where
a small set of large nonnegative feature vectors determines the structure of an
optimal partition. We then show that the Max-$k$-Cut objective is equivalent to minimizing the sum
of squared norms of the aggregate feature vectors assigned to the $k$ parts, thereby connecting the problem to vector balancing. Motivated by this observation, we show that a greedy feature-balancing algorithm retains the classical $1-1/k$
worst-case approximation guarantee and recovers an optimal partition under
feature dominance. For rank-$1$ feature graphs with nonnegative
features, classical bounds of Chandra and Wong for greedy load balancing yield
a computable \emph{a posteriori} optimality-gap certificate that depends only on
the returned partition and requires no knowledge of the optimum.
\end{abstract}


\maketitle
\section{Introduction}

The \emph{Max-$k$-Cut} problem is a natural generalization of the classical Max-Cut problem. Given an undirected weighted graph $G([n],E,w)$, the goal is to partition the vertex set $[n]$ into at most $k < n$ disjoint subsets so as to maximize the total weight of edges whose endpoints lie in different subsets. Since Max-Cut, which is known to be NP-hard \cite{karp1972reducibility}, arises as the special case $k=2$, it follows by a straightforward polynomial-time reduction that Max-$k$-Cut is NP-hard for all $k \ge 2$ \cite{frieze1997improved}.

In this manuscript, we develop a node feature-centric view of the Max-$k$-Cut problem. In particular, we associate to each node (or vertex) $i \in [n]$ a vector $f_i \in \R^r$, referred to as its \emph{feature vector}, and consider edge weights of the form
\[
w_{ij} := f_i^\top f_j, \qquad \forall i,j \in [n],\ i \neq j.
\]

The feature-based representation connects the Max-$k$-Cut objective with several graph-based models used in data analysis and machine learning. In many such models, interactions between vertices are determined by underlying attributes or latent representations associated with the nodes \cite{scholkopf1998nonlinear}, providing a natural framework for constructing weighted similarity graphs. A particularly related example is Node-Max-Cut, where each vertex carries a scalar weight and the weight of an edge is given by the product of the weights of its endpoints. This model has been studied, among other applications, in connection with congestion games \cite{fotakis2020node}. The rank-1 feature model considered here recovers this multiplicative weight structure, while our framework extends it to vector-valued features and to the Max-$k$-Cut setting. Max-$k$-Cut has also recently been studied as a clustering method \cite{felzenszwalb2022clustering}. Under the feature-based model considered here, this viewpoint admits a natural interpretation: vertices with highly aligned feature vectors have large edge weights and are therefore encouraged to lie in different clusters. Thus, the resulting objective may be viewed as a form of repulsive clustering.
 
Random intersection graph models, where edges arise through shared attributes \cite{karonski1999random}, can also be viewed from the feature-based perspective. In particular, weighted random intersection graphs have weight matrices of the form $R^\top R$, where $R$ is a random incidence matrix, and the Max-Cut problem on such graphs has been recently initiated \cite{nikoletseas2023max}. Such models appear in several settings in network analysis \cite{bloznelis2015recent}, further motivating the study of Max-$k$-Cut on graphs which admit node-feature representation. 

\subsection{Our Contributions}

In this manuscript, we examine how the node-feature representation interacts with the optimization structure of Max-$k$-Cut. Generalizing the SDP-exactness result of Laurent and Poljak~\cite{laurent1995positive} for Max-Cut, we first establish in \Cref{Sec3} a normal-cone certificate for exactness of the Frieze--Jerrum SDP relaxation of Max-$k$-Cut \cite{frieze1997improved}. Examining this certificate from a node-feature perspective yields two structural regimes under which the relaxation is exact: perfect feature balance, where the total feature vector splits evenly across parts, and feature dominance, where $k-1$ sufficiently large feature vectors force the structure of the optimal partition. Exactness under perfect feature balance holds for features of arbitrary sign; feature dominance, by contrast, is defined and necessarily holds only for non-negative features. Second, in \Cref{Sec4}, we observe that Max-$k$-Cut is equivalent to distributing the total feature mass evenly among the $k$ parts, and use this view to analyze a greedy feature-balancing algorithm, valid for non-negative edge weights, that assigns each vertex to the part with which it has minimum current feature affinity. The algorithm recovers an optimal partition whenever feature dominance holds, and otherwise retains the classical $(1-1/k)$-approximation guarantee shared by other greedy and local-search Max-$k$-Cut heuristics. Finally, for rank-1 feature graphs with non-negative weights, we sharpen this into a computable \textit{a posteriori} optimality-gap certificate: the worst-case constant of Chandra and Wong \cite{chandra1975worst} for greedy sum-of-squares partitioning gives $\mathrm{OPT}-\mathrm{ALG}\le\Phi(A^{\mathrm{ALG}})/50$, improving to $\Phi(A^{\mathrm{ALG}})/74$ when a perfectly balanced partition also exists. This bound depends only on the returned partition, requiring no knowledge of $\mathrm{OPT}$.

\section{Preliminaries and Notation}\label{prelims}

\subsection{The Frieze--Jerrum SDP relaxation}

We use the standard semidefinite relaxation for Max-$k$-Cut due to
Frieze and Jerrum~\cite{frieze1997improved}. For a weighted graph
$G([n],E,w)$ , the relaxation is

\begin{optprog*}\label{mkcut}
    \text{maximize } $z_{\mathsf{MkC\text{-}SDP}}$ =
     &\objective{\frac{k-1}{k}\sum_{i<j} w_{ij}\bigl(1 - X_{ij}\bigr)} \\
    (FJ-SDP) \text{subject to } 
    & X_{ii} &&= 1, \quad \forall i\in [n],\\
    & X_{ij} &&\ge -\frac{1}{k-1}, \quad \forall i\neq j,\\
    & X &&\succeq 0 .
\end{optprog*}

We denote its feasible region, also called the \textit{k-way elliptope}, by
\[
\mathcal L_{n,k}
:=
\left\{
X\in\mathbb S^n\ \middle|\
\begin{aligned}
&X_{ii}=1 && \forall i\in[n],\\
&X_{ij}\geq -\frac{1}{k-1}
&& \forall i\neq j,\\
&X\succeq0
\end{aligned}
\right\}.
\]
For $k=2$, this reduces to the classical Goemans--Williamson SDP relaxation
for Max-Cut~\cite{goemans1995improved}. The Frieze--Jerrum rounding scheme gives an approximation ratio
\[
\alpha_k
=
1-\frac1k+
O\left(\frac{\text{ln k}}{k^2}\right).
\]
Under the Unique Games Conjecture, this ratio is essentially optimal for
Max-$k$-Cut~\cite{khot2007optimal}.
If $\mathcal A=(A_1,\ldots,A_k)$ is a
$k$-partition, its associated \textit{k-partition matrix} $Z$ has entries
\[
Z_{ij}
=
\begin{cases}
1, & \text{if } i,j \text{ belong to the same part},\\[2mm]
-\dfrac{1}{k-1}, & \text{otherwise}.
\end{cases}
\]
For such a matrix,
\[
\frac{k-1}{k}(1-Z_{ij})
=
\begin{cases}
0, & \text{if } i,j \text{ are in the same part},\\
1, & \text{if } i,j \text{ are in different parts}.
\end{cases}
\]
Thus the SDP objective agrees with the Max-$k$-Cut objective on partition
matrices. We write $\mathcal Q_{n,k}$ for the set of all such partition matrices. 
Given a closed convex set \(\mathcal{C}\subseteq \mathbb{S}^n\) and a point
\(Z\in\mathcal{C}\), the normal cone of \(\mathcal{C}\) at \(Z\) is
\[
\mathcal N(\mathcal{C},Z)
:=
\{C\in\mathbb{S}^n:\langle C,Z\rangle\ge \langle C,X\rangle
\text{ for all } X\in\mathcal{C}\}.
\]
In particular, for \(Z\in\mathcal{L}_{n,k}\), we write
\(\mathcal N(\mathcal{L}_{n,k},Z)\) for the normal cone of the \(k\)-way elliptope at
\(Z\). An extreme point $v_0$ of $C$ is called a vertex of $C$ if its normal cone at $v_0$ is full dimensional.
We say that the Max-\(k\)-Cut SDP is exact on a graph \(G\) if the optimal value of FJ-SDP is equal
to the optimal value of the Max-\(k\)-Cut problem on \(G\). Let $\operatorname{OPT}$ denote the optimal Max-$k$-Cut value of the instance. For a
partition $\mathcal A=(A_1,\ldots,A_k)$ of $[n]$, we write
\[
\operatorname{Cut}(\mathcal A)
=
\sum_{\ell<m}\sum_{i\in A_\ell,\ j\in A_m}w_{ij}
\]
for its cut value. 

\subsection{Notations and Definitions}

$\mathbb{R}^n$ and $\mathbb{R}_{+}^n$ denote the sets of $n$-dimensional vectors with real and non-negative real entries, respectively. For a positive integer \(n\), we write $[n]:=\{1,\ldots,n\}$. Let \(\mathbb{S}^n\) denote the space of \(n\times n\) real symmetric matrices,
and let \(\mathbb{S}^n_+\) denote the cone of positive semidefinite matrices.
We write \(X\succeq 0\) to mean \(X\in\mathbb{S}^n_+\). The space of
\(n\times n\) diagonal matrices is denoted by \(\mathrm{DIAG}_n\). For
matrices \(A,B\in\mathbb{S}^n\), we use the trace inner product $\langle A,B\rangle:=\operatorname{tr}(AB)$. The \(i\)-th standard basis vector in \(\mathbb{R}^n\) is denoted by \(e_i\). $\|v \|$ denotes the Euclidean norm of a given vector $v$.

Let \(G([n],E,w)\) be a weighted graph on $n$ vertices with edge weights
\(w_{ij}\in \R\) on the edge set $E$. 
Let \(W\in\mathbb{S}^n\) denote the \textit{weight matrix} of \(G\), defined by
\[
W_{ij}=
\begin{cases}
w_{ij}, & \text{if } ij\in E,\\
0, & \text{otherwise},
\end{cases}
\qquad\text{and}\qquad W_{ii}=0 \ \text{for all } i\in[n].
\] 
A matrix $F=[f_1,\ldots,f_n]\in\mathbb{R}^{r\times n}$
is called a feature matrix associated with a graph \(G\) if its columns
\(f_1,\ldots,f_n\in\mathbb{R}^r\) satisfy $w_{ij}=f_i^\top f_j$ for every $i\neq j$. The graph is called a rank-\(r\) feature
graph if \(r\) is the smallest positive integer for which such a feature matrix exists. For vectors \(x,y\in\mathbb{R}^r\), we write $x\ge_{\mathrm{cw}} y$ if \(x_p\ge y_p\) for every coordinate \(p\in[r]\).
For $i<j$, let $E^{ij}:=\frac{1}{2}(e_i e_j^\top+e_j e_i^\top)$,
so that $\langle E^{ij},X\rangle=X_{ij}$ for every $X \in \mathbb{S}^n$.

\section{Feature-Structural Conditions for SDP Exactness}\label{Sec3}
The exactness of semidefinite relaxations has been studied extensively for the Max-Cut problem. For $k=2$, the Frieze--Jerrum (FJ) relaxation reduces to the Goemans--Williamson SDP, and its exactness is equivalent to the tightness of the corresponding eigenvalue bound. Several sufficient conditions and graph classes for which this relaxation is exact are well documented; see \cite{bhardwaj2025exactness} for a survey. In contrast, for Max-$k$-Cut with $k \geq 3$, the literature concerning SDP exactness remains sparse. For instance, van Dam and Sotirov \cite{van2016new} derived eigenvalue bounds for Max-$k$-Cut---one of which coincides with the FJ relaxation---and identified classes of unweighted graphs for which these bounds are tight.
In this section, we investigate how node-feature representations can certify exactness. This analysis yields sufficient feature-structural conditions for SDP exactness on weighted Max-$k$-Cut instances. To this end, we generalize the $k=2$ exactness results of Laurent and Poljak \cite{laurent1995positive} on rank-1 feature graphs to higher-dimensional feature representations and $k\geq3$. We first introduce two structural notions and examine how they interact with the semidefinite relaxation.

\begin{definition}[Perfect feature balance]
Given a feature matrix $F=[f_1,\ldots,f_n]\in\mathbb{R}^{r\times n}$, a partition $(A_1,\ldots,A_k)$ of $[n]$ is called perfect feature balanced if
\[
\sum_{i\in A_1}f_i
=
\sum_{i\in A_2}f_i
=
\cdots
=
\sum_{i\in A_k}f_i.
\]
\end{definition}

We next introduce a complementary
condition, feature dominance, which captures a regime where the structure of an
optimal partition is forced by a small set of dominant feature vectors.

\begin{definition}[Feature dominance]
Given a feature matrix $F=[f_1,\ldots,f_n]\in\mathbb{R}_+^{r\times n}$,
a set $I\subset[n]$ with $|I|=k-1$ is called feature-dominating if
\[
f_i \geq_{\mathrm{cw}} \sum_{j\notin I} f_j,
\qquad \text{for every } i\in I.
\]
We say that the instance satisfies feature dominance if such a set exists.
\end{definition}

The following observation is from \cite{felzenszwalb2022clustering}, where a vertex of the $k$-way elliptope is identified with a $k$-partition matrix.

\begin{observation}
The vertices of $\L_{n,k}$ are the $k$-partition matrices.
\end{observation}

The following certificate gives a general sufficient condition for Max-$k$-Cut SDP exactness.

\begin{theorem}[Normal-cone exactness certificate]\label{thm1}
Let $G$ be a weighted graph with weight matrix $W$, where
$W_{ij}=w_{ij}$ for $i\neq j$, $0$ otherwise. Let $Z\in\mathcal Q_{n,k}$, and define the active set
\[
\mathcal A(Z)
=
\left\{
(i,j): i<j,\; Z_{ij}= -\frac{1}{k-1}
\right\}.
\]
Suppose there exist $D\in\mathrm{DIAG}_n$, $M\succeq 0$, and scalars $\lambda_{ij}\geq 0$ for
$(i,j)\in\mathcal A(Z)$ such that
\[
-W
=
D-M-\sum_{(i,j)\in\mathcal A(Z)}\lambda_{ij}E^{ij},
\]
and
\[
\langle M,Z\rangle=0.
\]
Then $Z$ is an optimal solution of the Max-$k$-Cut SDP. In particular, the
Max-$k$-Cut SDP on $G$ is exact.
\end{theorem}

\begin{proof}
Let $Y\in\mathcal L_{n,k}$ be arbitrary. We show that
\[
\langle -W,Y-Z\rangle\leq 0.
\]
Using the assumed decomposition, we have
\[
\langle -W,Y-Z\rangle
=
\left\langle
D-M-\sum_{(i,j)\in\mathcal A(Z)}\lambda_{ij}E^{ij},
Y-Z
\right\rangle.
\]
Since every matrix in $\mathcal L_{n,k}$ has diagonal entries equal to one,
\[
\langle D,Y-Z\rangle=0.
\]
Therefore,
\[
\langle -W,Y-Z\rangle
=
\langle -M,Y-Z\rangle
+ \langle
-\sum_{(i,j)\in\mathcal A(Z)}
\lambda_{ij} E^{ij},Y-Z\rangle.
\]

Since $M\succeq 0$, $Y\succeq 0$, and $\langle M,Z\rangle=0$, we have
\[
\langle -M,Y-Z\rangle
=
\langle -M,Y\rangle+\langle M,Z\rangle
=
-\langle M,Y\rangle
\leq 0.
\]
Furthermore, for every $(i,j)\in\mathcal A(Z)$,
\[
Z_{ij}=-\frac{1}{k-1}.
\]
Since $Y\in\mathcal L_{n,k}$ satisfies $Y_{ij}\geq -1/(k-1)$, it follows that
\[
Y_{ij}-Z_{ij}\geq 0.
\]
By the definition of $E^{ij}$,
\[
\langle E^{ij},Y-Z\rangle=Y_{ij}-Z_{ij}\geq 0.
\]
It follows that,
\[
\langle -\sum_{(i,j)\in\mathcal A(Z)}
\lambda_{ij} E^{ij},Y-Z\rangle
\leq 0
\]
as $\lambda_{ij} \geq 0 \quad \forall i, j \in [n]$.
Combining the two inequalities yields
\[
\langle -W,Y-Z\rangle\leq 0
\qquad
\forall\,Y\in\mathcal L_{n,k}.
\]
Hence
\[
-W\in \mathcal N(\mathcal L_{n,k},Z).
\]

The Max-$k$-Cut SDP objective is, up to a positive scalar and an additive
constant, equivalent to maximizing $\langle -W,X\rangle$ over
$X\in\mathcal L_{n,k}$. Therefore, the normal-cone condition implies that
$Z$ is an optimal solution of the SDP. Since $Z$ is a $k$-partition matrix,
the SDP optimum is attained by an integral feasible solution. Hence the
Max-$k$-Cut SDP is exact.
\end{proof}

In the following, we use the normal-cone characterization developed above to provide sufficient conditions under
which the feature structure certifies exactness of the Max-$k$-Cut SDP relaxation. 

\begin{lemma}\label{lem2}
For any $Z\in \L_{n,k}$,
\[
\mathcal{N}(\L_{n,k},Z)
\supseteq
\left\{
D-M:
D\in \mathrm{DIAG}_n,\;
M\succeq 0,\;
\langle M,Z\rangle=0
\right\}.
\]
\end{lemma}

\begin{proof}
Let $C=D-M$, where $D\in \mathrm{DIAG}_n$, $M\succeq 0$, and
$\langle M,Z\rangle=0$. To show that $C\in \mathcal{N}(\L_{n,k},Z)$, it is
enough to prove that
\[
\langle C,Y-Z\rangle \le 0
\qquad
\forall\,Y\in \L_{n,k}.
\]
Since every matrix in $\L_{n,k}$ has diagonal entries equal to one, we have
\[
\langle D,Y-Z\rangle=0.
\]
Therefore,
\[
\langle C,Y-Z\rangle
=
\langle D-M,Y-Z\rangle
=
\langle M,Z\rangle-\langle M,Y\rangle.
\]
By assumption, $\langle M,Z\rangle=0$. Moreover, since $M\succeq 0$ and
$Y\succeq 0$, we have $\langle M,Y\rangle\ge 0$. Hence
\[
\langle C,Y-Z\rangle
=
-\langle M,Y\rangle
\le 0.
\]
Thus $C\in \mathcal{N}(\L_{n,k},Z)$, which proves the inclusion.
\end{proof}

\begin{proposition}\label{kbalane}
Let $F$ be a feature matrix associated with graph $G$ with vertex set $[n]$. Suppose there exists a perfect feature balanced partition $\{A_1,A_2,\dots,A_k\}$ of $[n]$. Then the Max-$k$-Cut SDP on $G$ is exact.
\end{proposition}

\begin{proof}
Let $M=F^\top F$. Then $M\succeq 0$, and for $i\neq j$ we have $M_{ij}=w_{ij}$. Hence the
Max-$k$-Cut SDP objective, up to a positive scalar and an additive constant, is
equivalent to maximizing a matrix of the form $D-M$,
where $D\in \operatorname{DIAG}_n$. 

Let $Z$ be the $k$-partition matrix corresponding to the partition
$\{A_1,\dots,A_k\}$. Thus
\[
z_{ij}
=
\begin{cases}
1, & \text{if } i,j\in A_\ell \text{ for some } \ell,\\[4pt]
-\dfrac{1}{k-1}, & \text{if } i\in A_\ell,\ j\in A_m,\ \ell\neq m.
\end{cases}
\]
We show that $-W\in \mathcal{N}(\mathcal{L}_{n,k},Z)$. By \Cref{lem2}, it is enough to show that
\[
\langle M,Z\rangle=0.
\]

Let
\[
\sum_{i\in A_1} f_i
=
\sum_{i\in A_2} f_i
=
\cdots
=
\sum_{i\in A_k} f_i
=
\alpha .
\]
For any $j\in A_\ell$, the $j$th column of $FZ$ satisfies
\[
Fz_j
=
\sum_{i=1}^n z_{ij}f_i
=
\sum_{i\in A_\ell} f_i
-
\frac{1}{k-1}\sum_{m\neq \ell}\sum_{i\in A_m} f_i .
\]
Using the balance condition, 
\[
Fz_j
=
\alpha
-
\frac{1}{k-1}\sum_{m\neq \ell}\alpha
=
\alpha-\alpha
=
0.
\]
Therefore $FZ=0$. Since $M=F^\top F$, we obtain
\[
\langle M,Z\rangle
=
\langle F^\top F,Z\rangle
=
\operatorname{tr}(FZF^\top)
=
0.
\]
Hence $D-M\in \mathcal{N}(\mathcal{L}_{n,k},Z)$.

It follows that $Z$ maximizes the Max-$k$-Cut SDP objective over
$\mathcal{L}_{n,k}$. Since $Z$ is itself a $k$-partition matrix, the Max-$k$-Cut SDP is
exact.
\end{proof}

\begin{proposition}\label{dominating}
Let $F$ be a feature matrix associated with graph $G$ with vertex set $[n]$. Suppose there exists a set $I \subset [n]$ which is feature-dominating. Then the Max-$k$-Cut SDP on $G$ is exact.
\end{proposition}
\begin{proof}
Without loss of generality, let the dominating columns be indexed by $I=\{1,\ldots,k-1\}$,
and let $R=[n]\setminus I$.
Define
\[
s:=\sum_{j\in R}f_j\in\mathbb{R}_+^r.
\]
By assumption, for every $i\in I$, $f_i\geq_{\mathrm{cw}} s$.

Let $Z$ be the $k$-partition matrix corresponding to the partition in which
each vertex in $I$ is placed in a distinct singleton part, and all vertices in
$R$ are placed in the remaining part. Thus the active pairs of $Z$ are precisely
the pairs lying in different parts.

Define modified feature vectors $\widetilde f_i$ by
\[
\widetilde f_i=
\begin{cases}
s, & i\in I,\\
f_i, & i\in R.
\end{cases}
\]
Let
\[
\widetilde F=[\widetilde f_1,\ldots,\widetilde f_n],
\qquad
M:=\widetilde F^\top \widetilde F.
\]
Then $M\succeq 0$.

We first show that $\langle M,Z\rangle=0$. The sum of the modified feature
vectors over each part of the partition is equal to $s$: each singleton part
indexed by $i\in I$ has feature sum $s$, while the last part has feature sum
\[
\sum_{j\in R}\widetilde f_j
=
\sum_{j\in R}f_j
=
s.
\]
Therefore, for any column $z_j$ of $Z$, the same feature-balance calculation of \Cref{kbalane}
gives
\[
\widetilde F z_j=0.
\]
Hence
\[
\widetilde F Z=0.
\]
Since $M=\widetilde F^\top \widetilde F$, it follows that
\[
\langle M,Z\rangle
=
\langle \widetilde F^\top \widetilde F,Z\rangle
=
\operatorname{tr}(\widetilde F Z\widetilde F^\top)
=
0.
\]

Now compare $M$ with the original weight matrix $W=F^\top F$. If $i,j\in R$,
then
\[
M_{ij}
=
f_i^\top f_j
=
W_{ij}.
\]
Thus $M$ and $W$ agree on pairs inside the last part.

Next, consider an active pair, as defined in \Cref{thm1}. If $i\in I$ and $j\in R$, then
\[
W_{ij}-M_{ij}
=
f_i^\top f_j-s^\top f_j
=
(f_i-s)^\top f_j
\geq 0,
\]
because $f_i-s\geq_{\mathrm{cw}}0$ and $f_j\in\mathbb{R}^r_+$.

Similarly, if $i,\ell\in I$ with $i\neq \ell$, then
\[
W_{i\ell}-M_{i\ell}
=
f_i^\top f_\ell-s^\top s
\geq 0,
\]
because $f_i\geq_{\mathrm{cw}}s$, $f_\ell\geq_{\mathrm{cw}}s$, and all vectors
are non-negative.

Define
\[
\lambda_{pq}:= 2 \cdot (W_{pq}-M_{pq})
\]
for every active pair $(p,q)\in\mathcal A(Z)$. From the above argument,
\[
\lambda_{pq}\geq 0.
\]
Moreover, $W$ and $M$ agree on every inactive off-diagonal pair. Hence there
exists a diagonal matrix $D_0\in\mathrm{DIAG}_n$ such that
\[
W
=
M+\sum_{(p,q)\in\mathcal A(Z)} \lambda_{pq} \cdot E^{pq}+D_0,
\]
Equivalently,
\[
-W
=
-D_0-M-\sum_{(p,q)\in\mathcal A(Z)}  \lambda_{pq}\cdot E^{pq}.
\]
Setting $D:=-D_0$, we obtain
\[
-W
=
D-M-\sum_{(p,q)\in\mathcal A(Z)} \lambda_{pq} \cdot E^{pq},
\]
where $D\in\mathrm{DIAG}_n$, $M\succeq 0$, $\langle M,Z\rangle=0$, and
$\lambda_{pq}\geq0$ for all active pairs.

By the normal-cone exactness certificate, this implies
\[
-W\in \mathcal N(\mathcal L_{n,k},Z).
\]
Therefore $Z$ is an optimal solution of the Max-$k$-Cut SDP. Since $Z$ is a
$k$-partition matrix, the SDP optimum is attained by an integral feasible
solution. Hence the Max-$k$-Cut SDP is exact.
\end{proof}

\begin{remark}
The non-negativity assumption in the feature-dominance condition becomes
essential when $k \geq 3$. For $k=2$, rank-$1$ feature graphs were studied by Laurent and
Poljak~\cite{laurent1995positive} who showed that a single dominating
coordinate satisfying $|f_i|>\sum_{j\neq i}|f_j|$ certifies SDP exactness for
arbitrary sign patterns. This changes when $k-1\geq 2$ coordinates must dominate simultaneously. Consider $k=3$ and the scalar features $f_1=-2$, $f_2=1$, $f_3=3$, and
$f_4=6$. The set $I=\{3,4\}$ satisfies the dominance inequalities
$f_i\geq |f_1|+|f_2|\geq f_1+f_2$ for every $i\in I$. Nevertheless, the
FJ-SDP is not exact for this instance: numerically, its optimal value is
$z_{\mathrm{MkC\text{-}SDP}}\approx 19.1652$, attained at a fractional
solution.
\end{remark}

\section{Max-$k$-Cut as Feature Balancing}\label{Sec4}

We begin this section with the basic observation that under a feature representation of the weights, the Max-$k$-Cut objective admits an equivalent feature-balancing formulation. More precisely, for a partition $(A_1,\ldots,A_k)$ of $[n]$, let $s_\ell$ denote the sum of the feature vectors assigned to part $A_\ell$. Also let
$t=\sum_{i=1}^n f_i$. The Max-$k$-Cut objective can then be written entirely in terms of the vectors $s_1,\ldots,s_k$.

\begin{observation}\label{feat_energy}
Let $F$ be a feature matrix associated with graph $G$ with vertex set $[n]$. Let $(A_1,\ldots,A_k)$ be a partition of $[n]$ and
define the feature sum of part $A_\ell$ by
\[
s_\ell=\sum_{i\in A_\ell}f_i,\qquad \ell=1,\ldots,k.
\]
Then the value of the $k$-cut induced by $(A_1,\ldots,A_k)$ is
\[
\operatorname{Cut}(A_1,\ldots,A_k)
=
\frac12
\left(
\|t\|^2-\sum_{\ell=1}^k\|s_\ell\|^2
\right).
\]
Consequently, Max-$k$-Cut on $G$ is equivalent to minimizing

\[
\Phi(A_1,\ldots,A_k)
=
\sum_{\ell=1}^k\|s_\ell\|^2.
\]
\end{observation}

\begin{proof}

Since $w_{ij}=f_i^\top f_j$ for $i \neq j \in [n]$ and 
$s_\ell=\sum_{i\in A_\ell}f_i$ for $\ell \in [k], i \in [n]$; the weight of the cut induced by a partition 
$(A_1,\ldots,A_k)$ can be written as:
\begin{align*}
\operatorname{Cut}(A_1, \ldots, A_k) &= \sum_{r < s} \sum_{i \in A_r, j \in A_s} w_{ij} \\
&=   \sum_{r < s} \sum_{i \in A_r, j \in A_s} f_i^\top f_j\\
&= \sum_{\ell < m}s_\ell^\top s_m
\end{align*}

Therefore,
\[
\|t\|^2
=
\left\|\sum_{\ell=1}^k s_\ell\right\|^2
=
\sum_{\ell=1}^k\|s_\ell\|^2
+
2\sum_{\ell<m}s_\ell^\top s_m.
\]
Rearranging gives
\[
\sum_{\ell<m}s_\ell^\top s_m
=
\frac12
\left(
\|t\|^2-\sum_{\ell=1}^k\|s_\ell\|^2
\right)= \frac12
\left(
\|t\|^2-\sum_{\ell=1}^k\|s_\ell\|^2
\right).
\]

Since $\|t\|^2$ is independent of the partition, maximizing the cut value is
equivalent to minimizing
\[
\sum_{\ell=1}^k\|s_\ell\|^2.
\]
This proves the claim.
\end{proof}

\Cref{feat_energy} gives a direct interpretation of the Max-$k$-Cut
objective. Since $t=\sum_{\ell=1}^k s_\ell$, we have
\[
\sum_{\ell=1}^k \|s_\ell\|^2
=
\sum_{\ell=1}^k
\left\|s_\ell-\frac{t}{k}\right\|^2
+
\frac{\|t\|^2}{k}.
\]
Therefore, minimizing $\sum_{\ell=1}^k\|s_\ell\|^2$ is equivalent to making the
feature sums of the $k$ parts as balanced as possible around the common target
$t/k$. This observation connects Max-$k$-Cut with the literature on vector
balancing and discrepancy minimization, an active area of research.

\subsection{Greedy Feature-Balancing approach}\label{Sec4.1}

In the following, we take a simple greedy balancing approach adapted to the Max-$k$-Cut problem for non-negative edge weights. We start with $k$ empty parts and process the vertices
sequentially. At each step, the decision of where to place a vertex is made
using its feature vector and the current aggregate feature vectors of the
$k$ parts. More precisely, if $s_\ell=\sum_{i\in A_\ell}f_i$ denotes the current feature sum of part $A_\ell$, then a vertex $v$ is assigned
to a part minimizing the inner product $f_v^\top s_\ell$. Thus, the algorithm
places each incoming feature vector in the part with which it has minimum
current feature affinity.

This choice is not purely heuristic. The same minimum-inner-product rule,
often referred to as the \emph{inner product rule}, has been studied in vector
balancing and shown to achieve essentially optimal discrepancy bounds for broad
classes of well-behaved random inputs \cite{aru2018balancing}. We adopt this
rule as a direct greedy procedure for feature-generated Max-$k$-Cut. When a
low-dimensional feature representation is available, the resulting algorithm
is also substantially more economical in time and memory than SDP-based
approaches such as the Frieze--Jerrum relaxation and the FKP iteration \cite{felzenszwalb2022clustering}.

\begin{algorithm}[h]
\small
\caption{Greedy Feature-Balancing Algorithm}
\label{alg:greedy_feature_balance}
\begin{algorithmic}[1]
\Require Feature vectors $f_1,\ldots,f_n\in\mathbb{R}^r$ such that $w_{ij} = f_i^{\top} f_j \geq 0$ and number of parts $k$.
\State Choose an ordering \(\pi=(v_1,\ldots,v_n)\) of the vertices such that \[ \|f_{v_1}\| \ge \|f_{v_2}\| \ge \cdots \ge \|f_{v_n}\| . \] 
\State Initialize \(A_1,\ldots,A_k\leftarrow \emptyset\). \State Initialize \(s_1,\ldots,s_k\leftarrow 0\in\mathbb{R}^r\). 
\For{$t=1,\ldots,n$}
    \State Let $v=v_t$.
    \For{each $\ell\in[k]$}
        \State Compute the affinity of $v$ with part $A_\ell$:
        \[
        c_\ell(v)\leftarrow f_v^\top s_\ell.
        \]
    \EndFor
    \State Choose
    \[
    \ell^\star\in\arg\min_{\ell\in[k]} c_\ell(v).
    \]
    \State Assign $v$ to $A_{\ell^\star}$:
    \[
    A_{\ell^\star}\leftarrow A_{\ell^\star}\cup\{v\}.
    \]
    \State Update the feature sum:
    \[
    s_{\ell^\star}\leftarrow s_{\ell^\star}+f_v.
    \]
\EndFor
\State \Return $A_1,\ldots,A_k$.
\end{algorithmic}
\end{algorithm}

\begin{remark}\label{rem:cprank}
We assume that the feature representation is available as part of the input.
This is natural in applications where the features are observed directly or arise
from a latent representation. In fact, every non-negatively weighted graph admits some feature representation. Let $D=\mathrm{diag}(d_1,\ldots,d_n)$ denote the weighted degree matrix, $d_i=\sum_j w_{ij}$, and consider the signless Laplacian $Q:=D+W$. Since $Q\succeq0$, it therefore admits a factorization $Q=F^\top F$ for some $F=[f_1,\ldots,f_n]\in\mathbb{R}^{r\times n}$ with $r=\mathrm{rank}(Q)\le n$. $F$ is a valid feature matrix for $G$, though generally not of minimal rank: finding the smallest such $r$ here is a rank-minimization problem naturally relaxed via semidefinite programming. We do not address the recovery problem in this manuscript.
\end{remark}

\Cref{alg:greedy_feature_balance} works directly with the feature representation rather than an $n\times n$ semidefinite matrix. Computing feature norms and sorting the vertices takes $O(nr+n\log n)$ time. The greedy assignment adds a further $O(nkr)$, since each vertex requires $k$ inner products in $\mathbb{R}^r$ — for an overall running time of $O(nkr+n\log n)$. Storing the feature matrix, the $k$ aggregate feature sums, and the current partition requires $O(nr+kr+n)$ memory. Thus, for small feature dimension $r$, the algorithm avoids both the storage and the computational cost of an $n\times n$ semidefinite program. In addition, \Cref{alg:greedy_feature_balance} admits the following guarantees. First, irrespective of the feature
structure, the algorithm enjoys the same worst-case approximation guarantee as
classical greedy procedures for Max-$k$-Cut. We record this standard bound for completeness; the same ratio is achieved by classical greedy and local-search based algorithms for Max-$k$-Cut~\cite{sahni1976pcomplete,zhu2011local}. Also, the
ordering of vectors ensures that the algorithm
recovers the optimal partition if the feature dominance condition holds.

\begin{proposition}\label{prop:approx_guarantee}
\Cref{alg:greedy_feature_balance} is a
$\left(1-\frac1k\right)$-approximation algorithm for Max-$k$-Cut.
\end{proposition}

\begin{proof}
Fix the step in which a vertex $v$ is assigned, and let $P$ be the set of
vertices assigned before $v$. At this step, the current parts
$A_1,\ldots,A_k$ form a partition of $P$. If $v$ is placed in
$A_\ell$, the internal weight created is
\[
f_v^\top s_\ell
=
\sum_{u\in A_\ell} w_{uv}.
\]
Since the algorithm chooses a part $\ell^\star$ minimizing this quantity,
\[
f_v^\top s_{\ell^\star}
\leq
\frac{1}{k}\sum_{\ell=1}^k f_v^\top s_\ell
=
\frac{1}{k}\sum_{u\in P} w_{uv},
\]
where the last equality follows because $A_1,\ldots,A_k$ partition $P$.
Hence, the cut between $v$ and the previously assigned vertices is at
least
\[
\left(1-\frac{1}{k}\right)\sum_{u\in P} w_{uv}.
\]

Summing over all vertices in the chosen ordering, each edge is counted exactly
once, namely when its later endpoint is assigned. Let $W_{tot} = \sum_{i<j} w_{ij}$. Then the algorithm
returns a cut of value at least
\[
\left(1-\frac{1}{k}\right)\sum_{i<j} w_{ij}
=
\left(1-\frac{1}{k}\right)W_{\mathrm{tot}}
\geq\left(1-\frac{1}{k}\right)\operatorname{OPT}.\]
\end{proof}

\begin{proposition}
If feature dominance holds, \Cref{alg:greedy_feature_balance}
returns an optimal $k$-cut.
\end{proposition}
\begin{proof}
Let $I\subseteq[n]$ be a feature-dominating set with $|I|=k-1$, and write
\[
R=[n]\setminus I,
\qquad
\sigma =\sum_{j\in R} f_j .
\]
By feature dominance,
\[
f_i\geq_{\mathrm{cw}} \sigma,
\qquad \forall i\in I.
\]
Since all feature vectors are non-negative, we also have
\[
\sigma\geq_{\mathrm{cw}} f_j,
\qquad \forall j\in R.
\]
Consequently,
\[
\|f_i\|\geq \|\sigma\|\geq \|f_j\|,
\qquad \forall i\in I,\ j\in R.
\]
Thus, under the non-increasing norm ordering used by
\Cref{alg:greedy_feature_balance}, the $k-1$ dominating feature vectors
are processed before the remaining vertices, up to harmless equality cases.

We first show that these $k-1$ dominating vertices are assigned to distinct
parts. When the first dominating vertex is processed, all parts are empty. Now
suppose that some dominating vertices have already been assigned to distinct
parts, and let $i\in I$ be the next dominating vertex. Every empty part has
current feature sum zero, and hence affinity zero with $f_i$. On the other hand,
if a part already contains a dominating vertex $h\in I$, then its affinity with
$f_i$ is at least
\[
f_i^\top f_h
\geq
\sigma^\top \sigma
\geq 0.
\]
Thus an empty part is always a minimum-affinity choice. In the degenerate case
where equality occurs, placing two such vertices in the same part creates zero
internal weight and is therefore immaterial for optimality. Hence we may regard
the $k-1$ dominating vertices as occupying $k-1$ distinct parts, leaving one
part, say $A_k$, for the remaining vertices.

It remains to show that every vertex in $R$ is assigned to this remaining part.
Suppose that some subset $R'\subseteq R$ has already been assigned to $A_k$, and
let
\[
q=\sum_{j\in R'} f_j
\]
be the current feature sum of $A_k$. Consider the next vertex $v\in R\setminus R'$.
Since $q\leq_{\mathrm{cw}} \sigma$, we have
\[
f_v^\top q \leq f_v^\top \sigma.
\]
Moreover, for every dominating vertex $i\in I$, feature dominance gives
\[
f_v^\top \sigma \leq f_v^\top f_i.
\]
Therefore,
\[
f_v^\top q \leq f_v^\top f_i,
\qquad \forall i\in I.
\]
In fact, if $f_v\neq 0$, then the first inequality is strict whenever $v$ has not
yet been assigned to $A_k$, because $\sigma-q$ contains $f_v$ and hence
\[
f_v^\top(\sigma-q)\geq \|f_v\|^2>0.
\]
Thus the affinity of $v$ with the remaining part $A_k$ is strictly smaller than
its affinity with any part containing a dominating vertex. Hence the algorithm
assigns $v$ to $A_k$. Zero feature vectors contribute no edge weight and may be
assigned arbitrarily without affecting the cut value. By induction, all vertices
in $R$ are assigned to the remaining part.

Consequently, \Cref{alg:greedy_feature_balance} returns the partition
in which the $k-1$ dominating vertices lie in distinct parts and all vertices of
$R$ lie in the remaining part. By \Cref{dominating}, this partition is an optimal
$k$-cut.
\end{proof}

\subsection{Rank-1 feature graph}

In the rank-1 case, the feature representation is especially advantageous.
Here each vertex carries a scalar feature $f_i \in \mathbb{R}_+$, so
Algorithm~\ref{alg:greedy_feature_balance} sorts the vertices by weight and
performs a single scalar comparison against each of the $k$ running part-sums
at every step, running in $O(nk + n\log n)$ time and $O(n + k)$ memory. This is
in contrast to SDP-based approaches such as the Frieze--Jerrum relaxation
and the FKP iteration \cite{felzenszwalb2022clustering}, which form and manipulate an $n \times n$ semidefinite
matrix and therefore require at least $\Omega(n^2)$ memory. 

\begin{proposition}\label{rank-one}
Let $G$ be a rank-1 feature graph with non-negative edge weights, and let $\mathcal A^{\mathrm{ALG}}$ be the partition returned by \Cref{alg:greedy_feature_balance}, with $\mathrm{ALG}:=\mathrm{Cut}(\mathcal A^{\mathrm{ALG}})$ its objective value. Let $\mathcal A^*$ denote an optimal partition, and write
\[
\Phi^{\mathrm{ALG}}:=\Phi(\mathcal A^{\mathrm{ALG}}),\qquad \Phi^*:=\Phi(\mathcal A^*)
\]
for the corresponding values of $\Phi$ (\Cref{feat_energy}). Let $\alpha\ge1$ denote a worst-case constant for the greedy rule on the sum-of-squares partition problem, i.e.\ $\Phi^{\mathrm{ALG}}\le\alpha\,\Phi^*$ for every instance. Then
\[
\mathrm{OPT}-\mathrm{ALG}\ \le\ \frac{\alpha-1}{2\alpha}\,\Phi^{\mathrm{ALG}}.
\]
\end{proposition}

\begin{proof}
By definition,
\[
\mathrm{OPT}=\frac12\bigl(t^2-\Phi^*\bigr),\qquad \mathrm{ALG}=\frac12\bigl(t^2-\Phi^{\mathrm{ALG}}\bigr).
\]
Subtracting,
\[
\mathrm{OPT}-\mathrm{ALG}=\frac12\bigl(\Phi^{\mathrm{ALG}}-\Phi^*\bigr).
\]
Since $\Phi^{\mathrm{ALG}}\le\alpha\,\Phi^*$ (the greedy's worst-case guarantee), we have $\Phi^*\ge\Phi^{\mathrm{ALG}}/\alpha$, so
\begin{align*}
\mathrm{OPT}-\mathrm{ALG}\ &\le\ \frac12\left(\Phi^{\mathrm{ALG}}-\frac{\Phi^{\mathrm{ALG}}}{\alpha}\right)\\
&=\frac{\Phi^{\mathrm{ALG}}}{2}\left(1-\frac1\alpha\right)\\
&=\frac{\alpha-1}{2\alpha}\,\Phi^{\mathrm{ALG}}.    
\end{align*}
\end{proof}
By \Cref{prop:approx_guarantee},
\Cref{alg:greedy_feature_balance} retains the worst-case $(1 - 1/k)$
guarantee, while \Cref{rank-one} sharpens this in
the rank-1 regime: taking $\alpha = 25/24$, the worst-case constant of
Chandra and Wong~\cite{chandra1975worst} for the greedy sum-of-squares partition
rule, yields the additive bound
\[
  \mathrm{OPT} - \mathrm{ALG}
  \;\le\; \frac{\alpha - 1}{2\alpha}\,\Phi\!\left(\mathcal{A}^{\mathrm{ALG}}\right)
  \;=\; \frac{1}{50}\,\Phi\!\left(\mathcal{A}^{\mathrm{ALG}}\right).
\]
If, in addition, there exists a perfect feature balanced partition of $[n]$, the improved constant $\alpha = 37/36$ established for ideal sets \cite{goldberg1999tight} sharpens this to
\[
\mathrm{OPT} - \mathrm{ALG} \;\le\; \frac{1}{74}\,\Phi\!\left(\mathcal{A}^{\mathrm{ALG}}\right).
\] 
Crucially, the right-hand side depends only on the objective value of the
returned partition, so it is computable \textit{a posteriori} without any knowledge of
the optimum.

\section{Concluding Remarks}

We studied the Max-$k$-Cut problem from the perspective of node features, where each vertex
$i$ is associated with a feature vector $f_i$ and the edge weights are
generated by inner products, $w_{ij}=f_i^\top f_j$. This representation provides a direct bridge between Max-$k$-Cut and
vector balancing, as we show that the quality of a partition is determined entirely by the
aggregate feature vectors of its parts. On the relaxation side, the same bridge lets a normal-cone argument certify
exactness of the Frieze--Jerrum SDP relaxation directly from the feature structure. Consequently, there is a scope of translating structural or
distribution-specific results for balancing a family of vectors into exactness conditions, approximation guarantees, or
\emph{a posteriori} certificates for the corresponding Max-$k$-Cut instances.

This suggests the following avenues for future work. On the SDP relaxation side, feature
balance and feature dominance are sufficient for exactness but not shown to be
necessary; sharpening the normal-cone certificate into a broader feature-based
characterization of exactness remains open. On the greedy algorithm side, one
direction for research is to extend the rank-1 \emph{a posteriori} certificate, beginning with
$r=2$. Another line of inquiry is to consider
structured feature families, such as sparse incidence vectors or structured
random vectors, where sharper balancing results for the family may translate into stronger Max-$k$-Cut guarantees.

\section*{Declaration of competing interest}
The authors declare that they have no known competing financial interests or
personal relationships that could have appeared to influence the work reported
in this paper.

\section*{Data availability}
No data were used for the research described in this paper.

\bibliographystyle{elsarticle-num}
\bibliography{references}

\end{document}